\documentclass[11pt, reqno]{amsart}
\usepackage{amssymb}
\usepackage{bbm}
\usepackage{amsfonts}
\usepackage{amsmath}
\usepackage{leftidx}
\usepackage{extarrows}
\usepackage[all]{xy}
\usepackage{cases}
\usepackage{txfonts}
\usepackage{mathrsfs}
\usepackage{epic}
\usepackage{dsfont}

\newtheorem{proposition}{Proposition}[section]
\newtheorem{lemma}[proposition]{Lemma}

\newtheorem{theorem}[proposition]{Theorem}
\newtheorem{remark}[proposition]{Remark}

\theoremstyle{definition}
\newtheorem{definition}[proposition]{Definition}

\def\<{\leqslant}
\def\>{\geqslant}

\date{}

\begin{document}
\title[Notes on gamma invariants of finite dimensional Hopf algebras]{Notes on gamma invariants of finite dimensional Hopf algebras}
\author{Xinru Zhang}
\address{Department of Mathematics $\&$ Statistics, University of Hasselt/School of Mathematical Science, Yangzhou University,Yangzhou 225002, China}
\email{xinru.zhang@uhasselt.be}
\author{Libin Li}
\address{School of Mathematical Science, Yangzhou University,
	Yangzhou 225002, China}
\email{lbli@yzu.edu.cn}
\author{Yinhuo Zhang}
\address{Department of Mathematics $\&$  Statistics, University of Hasselt, Universitaire Campus, 3590 Diepenbeek, Belgium
	Yangzhou 225002, China}
\email{yinhuo.zhang@uhasselt.be}

\begin{abstract}
Let $H$ be a finite-dimensional, non-semisimple Hopf algebra over an algebraically closed field $\mathbf{k}$. This paper investigates the asymptotic behavior of the core of left $H$-modules through the lens of the gamma invariant $\gamma_{\mathfrak{X}}$ relative to a representation ideal $I_{\mathfrak{X}}$. We establish an equivalent characterization for the quotient of the Green ring $R_{\mathfrak{X}}$ to be a transitive fusion ring, demonstrating that transitivity is synonymous with the non-degeneracy of a naturally induced bilinear form and the collapse of the ideals $\mathcal{P}_{+}$, $\mathcal{P}_{-}$ and $I_{\operatorname{max}}$ into a single ideal. Furthermore, we prove that the Green ring exhibits the structure of a representation ring in the sense of Benson, provided that the square of the antipode is an inner automorphism and the equality $I_{\operatorname{max}}=I_{\operatorname{proj}}$ holds. As an explicit application of these frameworks, we analyze the Drinfeld double $D(H_4)$ of the Sweedler algebra, identifying an infinite family of distinct representation ideals and proving that the maximal gamma invariant $\gamma_{\operatorname{max}}$ induces a genuine ring homomorphism. Finally, for Hopf algebras of finite representation type under the assumption $\mathcal{P}_{+} = \mathcal{P}_{-} = I_{\operatorname{max}}$, we show that $\gamma_{\operatorname{max}}$ coincides precisely with the Frobenius--Perron dimension, and we explicitly compute the gamma invariants for the standard basis elements of the Green ring of the Taft algebra $H_n(q)$.
\end{abstract}
\maketitle

\section*{Introduction}\label{sec:intro}

The gamma invariant is a powerful asymptotic tool used to measure the ``size" of the non-projective component of a module as it is subjected to repeated tensor powers. This invariant is particularly significant in the representation theory of finite-dimensional Hopf algebras, where the category of modules is frequently non-semisimple, meaning that generic representations cannot be decomposed into a direct sum of simple components.

The study of representation rings, or Green rings, of finite-dimensional Hopf algebras provides a fundamental framework for understanding the structural properties of monoidal categories and their underlying representations \cite{Etingof1, Montgomery1, Radford1}. In scenarios where the underlying Hopf algebra $H$ is non-semisimple, the associated Green ring $r(H)$ encapsulates intricate homological and combinatorial data regarding both projective and non-projective modules \cite{Benson-Cohomology, Chen2, Li1, Wang-Green}.

A central challenge in this domain is quantifying the asymptotic growth and behavior of the non-projective portion of a module $V$ under successive tensor products. Following the seminal work of Benson and Symonds \cite{Benson-Symonds}, the gamma invariant $\gamma_{\operatorname{proj}}(x)$ was introduced to analyze the growth of the non-projective ``core" of a module relative to the projective ideal $I_{\operatorname{proj}}$ (the ideal generated by the isomorphism classes of projective modules). In \cite{Benson-Banach}, Benson extended this formulation by replacing $I_{\operatorname{proj}}$ with an arbitrary representation ideal $I_{\mathfrak{X}}$. This invariant is formally defined as the limit of the $n$-th root of the dimension of the core of the $n$-th tensor power of an element:
$$\gamma_{\mathfrak{X}}(x) := \lim_{n \to \infty} \sqrt[n]{\dim_{\mathbf{k}} \operatorname{core}_{\mathfrak{X}}(x^n)}.$$

As established in recent literature \cite{Benson-Banach}, the gamma invariant can be naturally interpreted as the spectral radius of the image of an element within the completion of a suitable quotient ring. This connection elegantly bridges the gap between algebraic representation theory and the classical theory of commutative Banach algebras. Furthermore, Benson proposed an axiomatic definition of a ``representation ring," characterizing the precise conditions under which these rings exhibit predictable behavior with respect to duality and tensor products.

Building upon the classification of representation ideals—such as the unique minimal projective ideal $I_{\operatorname{proj}}$ and the maximal representation ideal $I_{\operatorname{max}}$—the present paper explores the conditions under which the quotient of the Green ring constitutes a transitive ring. Specifically, we demonstrate that under certain natural hypotheses, such as the square of the antipode $S^2$ being an inner automorphism and the equality $I_{\operatorname{max}} = I_{\operatorname{proj}}$ holding, the Green ring of a Hopf algebra satisfies the axioms of a Benson representation ring. We further show that the gamma invariant induces a well-defined ring homomorphism, providing a powerful linear-algebraic tool for the study of non-semisimple Hopf algebras.

The paper is organized as follows:

 In Section~1, we review the foundational definitions of the Green ring $r(H)$ and its canonical duality operator. We formally define representation ideals $I_{\mathfrak{X}}$ and the gamma invariant $\gamma_{\mathfrak{X}}(x)$, which quantifies the asymptotic growth of the core of a module.
 
 In Section~2, we investigate the relationship between the bilinear form on the quotient ring $R_{\mathfrak{X}}$ and the structural properties of the algebra. We establish an equivalent characterization showing that the non-degeneracy of this bilinear form is synonymous with the transitivity of the quotient ring and the collapse of the representation ideals $\mathcal{P}_{+}$, $\mathcal{P}_{-}$, and $I_{\operatorname{max}}$ into a single ideal. Furthermore, we introduce the abstract framework of a Benson representation ring $R$. We prove that if the antipode squared $S^2$ is an inner automorphism and $I_{\operatorname{max}} = I_{\operatorname{proj}}$, the Green ring $r(H)$ satisfies Benson's axioms for a representation ring.

 In Section~3, we provide a concrete demonstration of these concepts using the Drinfeld double $D(H_4)$ of the Sweedler algebra. Utilizing Chen's classification \cite{Chen1, Chen2} of the indecomposable left $D(H_4)$-modules including simple, projective, string, and band modules, we explicitly show the existence of infinitely many distinct representation ideals situated strictly between the projective ideal and the maximal ideal. Crucially, we prove that for $D(H_4)$, the gamma invariant $\gamma_{\operatorname{max}}$ induces a genuine ring homomorphism.

 Finally, in Section~4, we focus our attention on Hopf algebras of finite representation type. We demonstrate that under the assumption $\mathcal{P}_{+} = \mathcal{P}_{-} = I_{\operatorname{max}}$, the gamma invariant coincides precisely with the Frobenius--Perron dimension. Moreover, we explicitly compute the gamma invariants for the standard bases in the Green ring of the Taft algebra $H_n(q)$.

\section{Preliminaries}\label{1}

Throughout this paper, we assume that $H$ is a finite dimensional   Hopf algebra over an algebraically closed field $\mathbf{k}$. Unless otherwise specified, all modules are assumed to be finite-dimensional left $H$-modules. For the foundational theory of Hopf algebras and their representations, we refer the reader to \cite{Montgomery1, Radford1}.

The \textit{Green ring} (or representation ring) $r(H)$ of $H$ is the abelian group generated by the isomorphism classes $[V]$ of modules $V$ in $\operatorname{mod}(H)$, subject to the relations $[U \oplus V] = [U] + [V]$, where $\operatorname{mod}(H)$ denotes the category of finite-dimensional left $H$-modules. The multiplication in $r(H)$ is induced by the internal tensor product, $[U][V] := [U \otimes V]$, rendering $r(H)$ an associative ring. The set $\operatorname{ind}(H)$, comprising the isomorphism classes of all indecomposable left $H$-modules, forms a $\mathbb{Z}$-basis for $r(H)$.

Let $Z$ be an indecomposable $H$-module. Following Auslander \cite{Auslander1}, we define the element $\delta_{[Z]} \in r(H)$ as follows:
\begin{itemize}
    \item If $Z$ is non-projective, $\delta_{[Z]} := [X] - [Y] + [Z]$, where $0 \to X \to Y \to Z \to 0$ is the unique almost split sequence ending in $Z$.
    \item If $Z$ is projective, $\delta_{[Z]} := [Z] - [\operatorname{rad} Z]$.
\end{itemize}
Let $*$ denote the involution on $r(H)$ induced by the linear duality functor, so that $[Z]^* = [Z^*]$. We employ the associative bilinear form introduced in \cite{Wang1}, defined by $([X], [Y]) = \dim_{\mathbf{k}} \operatorname{Hom}_H(X, Y^*)$ for any $H$-modules $X$ and $Y$. This map uniquely extends to a $\mathbb{Z}$-bilinear form on $r(H)$ that satisfies the identity $([X], [Y]) = ([Y]^{**}, [X])$ (see \cite[Lemma 4.1]{Wang1}). Recall that the evaluation map for $X$ is the canonical $H$-morphism $\operatorname{ev}_X : X^* \otimes X \to \mathbf{k}$ given by $\operatorname{ev}_X(f \otimes x) = f(x)$.

To simplify notation, we set $R := r(H)$ and denote its basis by $\{x_i\}_{i \in \mathfrak{I}} = \operatorname{ind}(H)$, where $x_0 = [\mathbf{k}] = \mathbf{1}$ represents the isomorphism class of the trivial left $H$-module $\mathbf{k}$. For each index $i \in \mathfrak{I}$, let $i^*$ denote the unique index such that $x_{i^*} = x_i^*$. Given any element $x = \sum_{i \in \mathfrak{I}} a_i x_i \in R$, we denote the coefficient $a_i$ by $[x : x_i]$. An element $x \in R$ is said to be \textit{non-negative} if $[x : x_i] \geq 0$ for all $i \in \mathfrak{I}$. The sets of non-negative and positive elements are denoted by $R_{\geq 0}$ and $R_{> 0}$, respectively. Following \cite{Wang1}, we define the subsets
\[
\mathcal{P}_+ := \mathbb{Z} \{ x_i \mid i \in \mathfrak{I}, [x_i x_i^* : \mathbf{1}] = 0 \} \quad \text{and} \quad \mathcal{P}_- := \mathbb{Z} \{ x_i \mid i \in \mathfrak{I}, [x_i^* x_i : \mathbf{1}] = 0 \}.
\]
As established in \cite{Wang1}, $\mathcal{P}_+$ and $\mathcal{P}_-$ form right and left ideals of $R$, respectively.

Benson \cite{Benson-Banach} introduced the concept of a representation ideal within this framework. We adapt this definition to $R$ as follows:

\begin{definition}\label{def:rep-ideal}
An ideal $I_{\mathfrak{X}}$ of $R$ is called a \textit{representation ideal} if it is spanned as a $\mathbb{Z}$-submodule by a subset of the basis $\{x_i\}_{i \in \mathfrak{X}}$ for some proper subset $\mathfrak{X} \subsetneq \mathfrak{I}$ and is closed under the duality operator $*$.
\end{definition}

We denote by $I_{\operatorname{proj}}$ the \textit{projective ideal} of $R$, which is spanned by the isomorphism classes of all projective left $H$-modules. A straightforward verification shows that the regular module class $x_{\operatorname{reg}}$ satisfies the relation $x x_{\operatorname{reg}} = \dim_{\mathbf{k}}(x) x_{\operatorname{reg}}$ for all $x \in R$. Consequently, since $H$ is assumed to be non-semisimple, $I_{\operatorname{proj}}$ is the unique minimal representation ideal of $R$. 

Let $I_{\operatorname{max}}$ denote the maximal representation ideal of $R$. It follows directly from the definitions of $\mathcal{P}_+$ and $\mathcal{P}_-$ that $I_{\operatorname{max}} \subseteq \mathcal{P}_+ \cap \mathcal{P}_-$. In particular, if $\mathcal{P}_+$ and $\mathcal{P}_-$ coincide, then both ideals are equal to $I_{\operatorname{max}}$. Let $I_{\mathbb{C},\mathfrak{X}} := \mathbb{C} \otimes_{\mathbb{Z}} I_{\mathfrak{X}}$. We define the quotient ring $R_{\mathfrak{X}} := R/I_{\mathfrak{X}}$ and its complexification $R_{\mathbb{C}, \mathfrak{X}} := \mathbb{C} \otimes_{\mathbb{Z}} R_{\mathfrak{X}}$. In particular, we adopt the following shorthand notation for the maximal and projective configurations:
\[
R_{\operatorname{max}} := R/I_{\operatorname{max}}, \quad R_{\operatorname{proj}} := R/I_{\operatorname{proj}},
\]
with $R_{\mathbb{C},\operatorname{max}}$ and $R_{\mathbb{C},\operatorname{proj}}$ denoting their respective complexifications.

Following Benson and Symonds \cite{Benson-Symonds}, we investigate the \textit{gamma invariant} of a non-semisimple Hopf algebra. Let $I_{\mathfrak{X}}$ be a representation ideal of $R$. For any non-negative element $x \in R_{\geq 0}$, we can uniquely decompose it as $x = x' + x''$, where $x' = \sum_{i \in \mathfrak{I} \setminus \mathfrak{X}} a_i x_i$ and $x'' = \sum_{i \in \mathfrak{X}} a_i x_i$. We define the \textit{core} of $x$ relative to $I_{\mathfrak{X}}$ by setting $\operatorname{core}_{\mathfrak{X}}(x) := x'$. Furthermore, let $c_n^{\mathfrak{X}}(x) := \dim_{\mathbf{k}}(\operatorname{core}_{\mathfrak{X}}(x^n))$, where the dimension of a core element is understood as the linear combination of the dimensions of its constituent modules. The \textit{gamma value} of $x$ is defined as the limit superior:
\[
\gamma_{\mathfrak{X}}(x) := \limsup_{n \to \infty} \sqrt[n]{c_n^{\mathfrak{X}}(x)}.
\]
For any $x \in R_{\geq 0} \setminus I_{\mathfrak{X}}$, the sequence $\{c_n^{\mathfrak{X}}(x)\}_{n=1}^{\infty}$ is submultiplicative. Hence, by Fekete's Lemma, the limit exists and satisfies $\gamma_{\mathfrak{X}}(x) = \lim_{n \to \infty} \sqrt[n]{c_n^{\mathfrak{X}}(x)}$. If $x \in R_{\geq 0} \cap I_{\mathfrak{X}}$, then $\gamma_{\mathfrak{X}}(x) = 0$. We denote these invariants by $\gamma_{\operatorname{proj}}(x)$ and $\gamma_{\operatorname{max}}(x)$ when specialized to the projective and maximal ideals, respectively.

Finally, let $A$ be a unital Banach algebra. For any element $a \in A$, the \textit{spectrum} $\operatorname{Spec}_A(a)$ is defined as the set of scalars $\lambda \in \mathbb{C}$ such that $a - \lambda \mathbf{1}$ is not invertible in $A$. The associated \textit{spectral radius} is denoted by $\rho(a) := \sup \{ |\lambda| \mid \lambda \in \operatorname{Spec}_A(a) \}$.


\section{Representation ideals and gamma invariants for Hopf algebras}\label{2}

Assume that $I_{\mathfrak{X}}$ is a representation ideal of $R$. For any $x \in R$, let $\overline{x}$ denote its image in the quotient ring $R_{\mathfrak{X}} = R/I_{\mathfrak{X} }$. Following \cite{Wang1}, we define a bilinear form on $R_{\mathfrak{X}}$ by:
\[ [\overline{x}, \overline{y}]_{\mathfrak{X}} := [xy : \mathbf{1}], \quad \text{for } x, y \in R. \]
This form is associative and satisfies the $*$-symmetry property: $[\overline{x}, \overline{y}]_{\mathfrak{X}} = [\overline{y}^*, \overline{x}^*]_{\mathfrak{X}}$.

The following lemma, which characterizes the structure constants, is adapted from Theorems 2.7 and 4.6, Lemma 4.2, and Corollary 4.7 in \cite{Wang1}. These results remain valid for finite-dimensional Hopf algebras of infinite representation type.

\begin{lemma}\label{2.1}
The following properties hold in $R$:
\begin{enumerate}
    \item If $[x_i x_j : \mathbf{1}] > 0$, then $j = i^*$, $i = i^{**}$, and $[x_i x_j : \mathbf{1}] = 1$.
    \item $(\delta_{x_i}^*, x_j) = \delta_{i,j}$, where $\delta_{i,j}$ is the Kronecker delta.
    \item $[x_i x_{i^*} : \mathbf{1}] \neq 0$ if and only if $\delta_{x_0}x_i = \delta_{x_i}$.
\end{enumerate}
\end{lemma}

\begin{lemma}\label{2.2}
The left and right radicals of the bilinear form $[-,-]_{\mathfrak{X}}$ are $\mathcal{P}_+/I_{\mathfrak{X}}$ and $\mathcal{P}_-/I_{\mathfrak{X}}$, respectively.
\end{lemma}

\begin{proof}
We characterize the left radical. If $x \in \mathcal{P}_+$, then $xy \in \mathcal{P}_+$ for any $y \in R$ (since $\mathcal{P}_+$ is a right ideal). Thus, $[\overline{x}, \overline{y}]_{\mathfrak{X}} = [xy : \mathbf{1}] = 0$, placing $\overline{x}$ in the left radical.
Conversely, if $x \notin \mathcal{P}_+$, then $x$ has a non-zero coefficient $\lambda_i$ for some $x_i$ with $[x_i x_i^* : \mathbf{1}] \neq 0$. Evaluating the form on $[\overline{x\vphantom{_i^\star}}, \overline{x_i^\star}]_{\mathfrak{X}} $ (where $\star$ is the inverse of $*$), and applying Lemma \ref{2.1}, we have:
\begin{align*}
[\overline{x\vphantom{_i^\star}}, \overline{x_i^\star}]_{\mathfrak{X}} 
&= [xx_i^\star : \mathbbm{1}] = (\delta_{x_0}^*, xx_i^\star) \\[1ex]
&= (\delta_{x_0}^* x, x_i^\star) = (x_i^{\star**}, \delta_{x_0}^* x) \\[1ex]
&= ((\delta_{x_0}x_i)^*, x) = (\delta_{x_i}^*, x) = \lambda_i \neq 0.
\end{align*}

Thus, $\overline{x}$ is not in the left radical.
\end{proof}

A $\mathds{Z}_+$-ring with basis $\mathds{I}$ is \textit{transitive} if for any $X, Z \in \mathds{I}$, there exist $Y_1, Y_2 \in \mathds{I}$ such that $Z$ appears with a nonzero coefficient in both $XY_1$ and $Y_2X$.

\begin{theorem}\label{2.3}
The following statements are equivalent:
\begin{enumerate}
    \item The bilinear form $[-,-]_{\mathfrak{X}}$ is non-degenerate.
    \item $\mathcal{P}_+ = \mathcal{P}_- = I_{\mathfrak{X}} = I_{\operatorname{max}}$.
    \item $R_{\mathfrak{X}}$ is transitive with respect to the basis $\{\overline{x_i}\}_{i \in \mathfrak{J} \setminus \mathfrak{X}}$.
\end{enumerate}
\end{theorem}

\begin{proof}
The equivalence of (1) and (2) follows directly from Lemma \ref{2.2}, as the non-degeneracy of a form is equivalent to its radicals being zero (or in this case, being contained within the quotiented ideal).

$(2) \Rightarrow (3)$: Assume $\mathcal{P}_+ = \mathcal{P}_- = I_{\mathfrak{X}}$. By Lemma \ref{2.1}, for any $\overline{x} \in R_{\mathfrak{X}}$, we have $\overline{x} = \overline{x}^{**}$. Define a linear map $\tau: R_{\mathfrak{X}} \to \mathbb{Z}$ by $\tau(\overline{x}) = (\delta_{x_0}^*, x)$. Then $\tau(\overline{x_i}) = \delta_{i,0}$. Lemma \ref{2.1} implies that $\tau(\overline{x_i}\overline{x_j}) = 1$ if $j=i^*$ and $0$ otherwise. For any $i, j \in \mathfrak{I} \setminus \mathfrak{X}$, since $[x_j x_j^* : \mathbf{1}] = 1$, the trivial module $\mathbf{1}$ is a summand of $x_j x_j^*$. Thus $x_i$ is a summand of $x_j x_j^* x_i$. This implies there exists some $k \in \mathfrak{I} \setminus \mathfrak{X}$ such that $x_k$ is a summand of $x_j^* x_i$ and $[x_j x_k : x_i] \neq 0$. This satisfies the condition for transitivity.

$(3) \Rightarrow (1)$: Suppose $R_{\mathfrak{X}}$ is transitive. For any $i \in \mathfrak{I} \setminus \mathfrak{X}$, there exists $j$ such that $[\overline{x_j}\overline{x_i} : \mathbf{1}] \neq 0$. By Lemma \ref{2.1}, $j$ must be $i^*$. If $\overline{x} = \sum \lambda_i \overline{x_i}$ is in the left radical, then for all $j$, $[\overline{x}, \overline{x_j}]_{\mathfrak{X}} = \lambda_{j^*} [x_{j^*} x_j : \mathbf{1}] = 0$. Since the structure constants are positive, we must have $\lambda_i = 0$ for all $i$, proving non-degeneracy.
\end{proof}
\vskip 3mm
We now recall the definition of representation rings introduced by Benson \cite{Benson-Banach}. Let $R$ be a commutative $\mathbb{Z}_+$-ring with basis $\{x_{i}\}_{i \in \mathfrak{J}}$ and identity $x_{0}=\mathbf{1}$. We call $R$ a representation ring in the sense of Benson if it satisfies the following conditions:
\begin{enumerate}
    \item[(\textit{i})] There is an involution $i \mapsto i^*$ on the index set $\mathfrak{J}$ that induces an anti-involution on $R$.
    \item[(\textit{ii})] If $[x_{i}x_{j}:\mathbf{1}] > 0$, then $j=i^{\ast}$.
    \item[(\textit{iii})] If $[x_{i}x_{i^{\ast}}:\mathbf{1}] = 0$, then $[x_{i}x_{i^{\ast}}x_{i}:x_{i}] \geq 2$.
    \item[(\textit{iv})] There exists a ring homomorphism $\dim: R \rightarrow \mathds{Z}$ such that $\dim(x_{i})=\dim(x_{i^{\ast}}) > 0$ for all $i \in \mathfrak{J}$.
    \item[(\textit{v})] There exists $\rho \in R_{\succ 0}$ such that $x\rho=(\dim x)\rho$ for all $x \in R$.
\end{enumerate}
Next, we show that the Green rings of certain special Hopf algebras fall into the Benson representation rings. In what follows, unless otherwise specified, $P$ denotes an indecomposable projective $H$-module.

\begin{lemma}\label{2.4}
The radical $\operatorname{rad}(P \otimes P^*)$ contains a submodule isomorphic to the trivial module $\mathbf{k}$.
\end{lemma}

\begin{proof}
Let $\{p_i\}$ be a basis for $P$ and $\{p_i^*\}$ the dual basis. Consider the element $x = \sum p_i \otimes p_i^* \in (P \otimes P^*)^H$. The submodule generated by $x$ satisfies $\mathbf{k} x \cong \mathbf{k}$. Since $P$ is injective (as $H$ is self-injective), each indecomposable component $P_i$ of $P \otimes P^*$ has a simple socle. Therefore, there exists a summand $P_s$ such that $\operatorname{soc}(P_s) \cong \mathbf{k}$. Because the trivial module $\mathbf{k}$ is not projective, $P_s \neq \operatorname{soc}(P_s)$. Noting that $\operatorname{soc}(\operatorname{rad}(P_s)) \subseteq \operatorname{soc}(P_s)$ and $\operatorname{rad}(P_s) \neq 0$, it follows that $0 \neq \operatorname{soc}(\operatorname{rad}(P_s)) \subseteq \operatorname{rad}(P_s) \subseteq \operatorname{rad}(P \otimes P^*)$.
\end{proof}

\begin{remark}\label{2.5}
 The assumption that $P$ is projective can be replaced by a weaker condition: $P \in \mathcal{P}_+$. As previously shown, $P \otimes P^*$ has an indecomposable summand $P_s$ containing a submodule $U \cong \mathbf{k}$. Suppose for contradiction that $U \not\subseteq \operatorname{rad}(P_s)$. Then there exists a maximal submodule $M$ of $P_s$ not containing $U$. Since $U$ is simple, $U \cap M = 0$ and $U + M = P_s$, yielding $P_s = U \oplus M$. The indecomposability of $P_s$ then forces $P_s = U \cong \mathbf{k}$, which explicitly contradicts $P \in \mathcal{P}_+$. Therefore, $U \subseteq \operatorname{rad}(P_s)$, and the conclusion follows.
\end{remark}
\begin{lemma}\label{2.6}
If the square of the antipode $S^2$ is an inner automorphism of $H$, then $\mathbf{k}$ is isomorphic to a composition factor of $(P \otimes P^*)/\operatorname{rad}(P \otimes P^*)$.
\end{lemma}

\begin{proof}
If $S^2(h) = uhu^{-1}$, the map $\theta: P \to P^{**}$ defined by $\theta(x)(f) = f(ux)$ is an isomorphism. This yields a surjection $\sigma: P \otimes P^* \to P^{**} \otimes P^* \to \mathbf{k}$. Since  $P \otimes P^*$ is a direct sum $\oplus_iP_i$ of indecomposable projective modules and $(P \otimes P^*)/\ker \sigma \cong \mathbf{k}$, there exists an indecomposable projective summand $P_s$ such that $P_s/\operatorname{rad}(P_s) \cong \mathbf{k}$.
\end{proof}

\begin{remark}\label{2.7}
To prove that the map $\sigma$ is an $H$-module surjection, we assume that $S^2$ is an inner automorphism. In fact, this can be replaced by the condition: the Green ring $R$ of $H$ is commutative. This follows from the fact that for any $H$-module $M$, the evaluation map $ev_M: M^* \otimes M \to \mathbf{k}$ given by $f \otimes x \mapsto f(x)$ is an $H$-module homomorphism.
\end{remark}

A celebrated theorem of Auslander and Carlson \cite[Proposition 4.9]{Auslander1-Carlson} asserts that for any indecomposable $KG$-module $M$ over a field $K$ of characteristic $p$ such that $p \mid \dim M$, the tensor product $M \otimes M^* \otimes M$ contains $M \oplus M$ as a direct summand. In what follows, we establish a Hopf-algebraic analogue of this fundamental result.

\begin{theorem}\label{2.8}
Let $H$ be a finite-dimensional non-semisimple Hopf algebra over an algebraically closed field $\mathbf{k}$ such that the square of its antipode, $S^2$, is an inner automorphism. Let $P$ denote the indecomposable projective $H$-module. Then $[P \otimes P^* \otimes P : P] \geq 2$.
\end{theorem}

\begin{proof}
Let $M = P \otimes P^*$ and $N = \operatorname{rad}(M)$. By Lemmas \ref{2.4} and \ref{2.6}, $N$ and $M/N$ both have $\mathbf{k}$ as a composition factor. Thus $M$ has at least two composition factors isomorphic to $\mathbf{k}$. Since $P$ is projective, tensoring a composition series of $M$ with $P$ splits: $M\otimes P \cong \bigoplus (M_i/M_{i+1} \otimes P)$, where the $M_i$ are the submodules in the series. Since at least two $M_i/M_{i+1} \cong \mathbf{k}$, we have $[M \otimes P : P] \geq 2$.
\end{proof}

\begin{theorem}\label{2.9}
Let $H$ be a finite-dimensional, non-semisimple Hopf algebra with a commutative Green ring $R$. If $S^2$ is an inner automorphism and $I_{\operatorname{max}}=I_{\operatorname{proj}}$, then $R$ is a representation ring in the sense of Benson.
\end{theorem}

\begin{proof}
Since $S^2$ is an inner automorphism, the map $*$ induces an anti-involution on $R$, satisfying Axiom (i) of the definition of a Benson representation ring. By Lemma \ref{2.1}, Axiom (ii) holds. Since $R$ is commutative, we have $\mathcal{P}_+ = \mathcal{P}_- = I_{\operatorname{max}}$. This equality, alongside Theorem \ref{2.8} (derived from Lemmas \ref{2.4} and \ref{2.6}) and the assumption $I_{\operatorname{max}}=I_{\operatorname{proj}}$, establishes Axiom (iii), namely that $[x_i x_{i^*} x_i : x_i] \geq 2$. The dimension is understood as the standard module dimension, hence Axiom (iv) holds. Finally, the regular module $x_{\operatorname{reg}}$ satisfies $x x_{\operatorname{reg}} = \dim_{\mathbf{k}}(x) x_{\operatorname{reg}}$, which verifies Axiom (v).
\end{proof}

\section{Representation Ideals and Gamma Invariants: The Case of $D(H_4)$}\label{3}

In this section, we study the Drinfeld double $D(H_4)$ of the 4-dimensional Sweedler algebra $H_4$ over an algebraically closed field $\mathbf{k}$ with $\operatorname{char}(\mathbf{k}) \neq 2$. We show that $D(H_4)$ admits infinitely many representation ideals, providing a unified notation. Furthermore, we demonstrate that when the maximal representation ideal is chosen, the corresponding gamma invariant induces a ring homomorphism.

As $D(H_4)$ is a quasitriangular Hopf algebra, its Green ring $r(D_4)$ is commutative. Based on the classification in \cite{Chen1}, the indecomposable $D_4$-modules are categorized as follows:
\begin{itemize}
    \item \textbf{Simple modules:} $V(r)$ and $V(2,r)$ for $r \in \mathbb{Z}_2$. Here, $V(0)$ denotes the trivial module $\mathbf{k}$, and the modules $V(2,r)$ are both projective and injective.
    \item \textbf{Projective modules:} $P(r)$, the projective covers of $V(r)$, for $r \in \mathbb{Z}_2$.
    \item \textbf{String modules:} $\Omega^s V(r)$ and $\Omega^{-s} V(r)$ for $s \geq 1$.
    \item \textbf{Band modules:} $M_s(r, \eta)$ for $s \geq 1, r \in \mathbb{Z}_2$, and $\eta \in \overline{\mathbf{k}} = \mathbf{k} \cup \{\infty\}$.
\end{itemize}

Recall from \cite{Chen2} that the dual modules are given by $V(r)^* \cong V(r)$, $V(2, r)^* \cong V(2, r+1)$, $P(r)^* \cong P(r)$, $(\Omega^{m} V(r))^* \cong \Omega^{-m} V(r)$ and $(M_s(r, \eta))^* \cong M_s(r+1, \eta)$. Furthermore, we have the following tensor product rules:
\begin{align}
    V(r) \otimes V(r') &\cong V(r+r'), \tag{3.1} \label{eq:1} \\
    V(r') \otimes M_s(r, \eta) &\cong M_s(r+r', \eta), \tag{3.2} \label{eq:2} \\
    M_s(r, \alpha) \otimes M_t(r', \eta) &\cong stP(r+r'), \tag{3.3} \label{eq:3} \\
    P(r') \otimes M_s(r, \eta) &\cong sP(0) \oplus sP(1), \tag{3.4} \label{eq:4} \\
    V(2, r') \otimes M_s(r, \eta) &\cong sV(2, 0) \oplus sV(2, 1), \tag{3.5} \label{eq:5} \\
    \Omega^m V(r) \otimes \Omega^n V(r') &\cong \Omega^{m+n}V(r+r') \oplus (\operatorname{proj}), \tag{3.6} \label{eq:6} \\
    M_s(r, \eta) \otimes M_t(r', \eta) &\cong s(t-1)P(r+r') \oplus M_s(0, \eta) \oplus M_s(1, \eta), \tag{3.7} \label{eq:7} \\
    M_s(r, \eta) \otimes \Omega^m V(r') &\cong 
    \begin{cases} 
        M_s(r+r', \eta) \oplus (\operatorname{proj}), & \text{if } m \text{ is even}, \\ 
        M_s(r+r'+1, \eta) \oplus (\operatorname{proj}), & \text{if } m \text{ is odd}, \tag{3.8} \label{eq:8}
    \end{cases}
\end{align}
where $m \in \mathbb{Z} \setminus \{0\}$, $1 \le s \le t$, and $\alpha, \eta \in \overline{\mathbf{k}}$ such that $\alpha \neq \eta$. Here, we define $\Omega^0 V(r) = V(r)$ and let $(\operatorname{proj})$ denote the projective parts of corresponding modules.

\begin{lemma}\label{lem:ideals}
For the Green ring $r(D_4)$, the projective and maximal representation ideals are given by:
\begin{align*}
    I_{\operatorname{proj}} &= \mathbb{Z}\{V(2,r), P(r) \mid r \in \mathbb{Z}_2\}, \\
    I_{\operatorname{max}} &= \mathbb{Z}\{V(2,r), P(r), M_s(r, \eta) \mid r \in \mathbb{Z}_2, s \geq 1, \eta \in \overline{\mathbf{k}}\}.
\end{align*}
\end{lemma}

\begin{proof}
The structure of $I_{\operatorname{proj}}$ follows from the fact that $D(H_4)$ is a Frobenius algebra, in which projective and injective modules coincide. For $I_{\operatorname{max}}$, notice that $(M_s(r, \eta))^* \cong M_s(r+1, \eta)$, and the equations \eqref{eq:6} and \eqref{eq:7} imply that for any band module $M = M_s(r, \eta)$, the trivial module $\mathbf{1}=V(0)$ is not a direct summand of $M \otimes M^*$. In contrast, for string modules $\Omega^s V(r)$ with $s \in \mathbb{Z} \setminus \{0\}$, the trivial module appears in the decomposition $\Omega^s V(r) \otimes (\Omega^s V(r))^*$. Thus, $M_s(r, \eta) \in \mathcal{P}_+ = \mathcal{P}_-$, whereas $\Omega^s V(r) \notin \mathcal{P}_+ = \mathcal{P}_-$. This implies that band modules belong to the maximal ideal, while string modules do not. This establishes the assertion for $I_{\operatorname{max}}$.
\end{proof}

\begin{lemma}\label{lem:all representation ideals}
In the Green ring $r(D_4)$, every representation ideal is of the form 
$$I_{\Theta}= \mathbb{Z}\{V(2, r), P(r), M_s(r, \eta) \mid r \in \mathbb{Z}_2, s \geq 1, \eta \in \Theta\},$$
for some subset $\Theta \subseteq \overline{\mathbf{k}}$.
\end{lemma}
\begin{proof}
First, let $\Theta \subseteq \bar{\mathbf{k}}$. Since $I_{\emptyset} = I_{\text{proj}}$ is the unique minimal representation ideal of $r(D_4)$, we have $I_{\text{proj}} \subseteq I_{\Theta}$. Equations (3.2)--(3.5) and (3.7)--(3.8), together with the properties of dual modules, ensure that $I_{\Theta}$ is closed under Green ring multiplication and the dual operator; hence, it is indeed a representation ideal.
Conversely, let $I$ be an arbitrary representation ideal. The natural inclusions $I_{\text{proj}} \subseteq I \subseteq I_{\text{max}}$ (where $I_{\text{max}} = I_{\bar{\mathbf{k}}}$) imply that $I$ contains all $V(2,r)$ and $P(r)$. Furthermore, the equations (3.2) and (3.7) imply that if $M_{s_1}(r, \eta) \in I$ for some $s_1 \ge 1$ and $r \in \mathbb{Z}_2$, then $M_s(0, \eta), M_s(1, \eta) \in I$ for all $s \ge 1$. By defining $\Theta = \{ \eta \in \bar{\mathbf{k}} \mid M_s(r, \eta) \in I \text{ for some } s \ge 1 \text{ and } r \in \mathbb{Z}_2 \}$, it immediately follows that $I = I_{\Theta}$.
\end{proof}

\begin{lemma}\label{lem:gamma_vals}
In the Green ring $r(D_4)$, the gamma invariants relative to $\mathfrak{X}_{\operatorname{max}}$ satisfy:
\begin{itemize}
    \item $\gamma_{\operatorname{max}}(V(r)) = \gamma_{\operatorname{max}}(\Omega^{\pm s} V(r)) = 1$ for all $s \geq 1$.
    \item $\gamma_{\operatorname{max}}(V(2,r)) = \gamma_{\operatorname{max}}(P(r)) = \gamma_{\operatorname{max}}(M_s(r, \eta)) = 0$.
\end{itemize}
\end{lemma}
\begin{proof}
Since $V(2,r)$, $P(r)$, and $M_s(r, \eta) \in I_{\operatorname{max}}$ by Lemma \ref{lem:ideals}, it follows that $\gamma_{\operatorname{max}}(V(2,r)) = \gamma_{\operatorname{max}}(P(r)) = \gamma_{\operatorname{max}}(M_s(r, \eta)) = 0$. For simple modules $V(r)$, since $\dim_{\mathbf{k}}(V(r)) = 1$, we have $\gamma_{\operatorname{max}}(V(r))=1$.  Moreover,  from \cite{Chen2}  we know that $\dim_{\mathbf{k}}(\Omega^m V(r)) = 2m + 1$ for all $m \in \mathbb{Z}$ and $r \in \mathbb{Z}_2$. Consequently,  by Equation \eqref{eq:6}, we have
$$
\begin{aligned}
\gamma_{\operatorname{max}}(\Omega^{\pm s} V(r)) 
&= \lim_{n \to \infty} \sqrt[n]{\dim_{\mathbf{k}}\operatorname{core}_{\mathfrak{X}_{\operatorname{max}}}((\Omega^{\pm s} V(r))^n)} \\
&= \lim_{n \to \infty} \sqrt[n]{\dim_{\mathbf{k}}(\Omega^{\pm ns} V(nr))} \\
&= \lim_{n \to \infty} \sqrt[n]{2ns+1} = 1.
\end{aligned}
$$
\end{proof}

\begin{theorem}\label{thm:homomorphism}
For the Green ring $r(D_4)$, the map $\gamma_{\operatorname{max}}$ induces a homomorphism from $r(D_4)_{\geq 0}$ to $\mathbb{Z}_{\geq 0}$. That is, for any $x, y \in r(D_4)_{\geq 0}$:
\[ \gamma_{\operatorname{max}}(x+y) = \gamma_{\operatorname{max}}(x) + \gamma_{\operatorname{max}}(y), \quad \gamma_{\operatorname{max}}(xy) = \gamma_{\operatorname{max}}(x)\gamma_{\operatorname{max}}(y). \]
\end{theorem}

\begin{proof}
Every element in $r(D_4)_{\ge 0}$ is a finite linear combination of indecomposable modules. For simplicity, let us consider the following example:
\[ x = aV(r_1) + b\Omega^{s_1} V(r_2) + c\Omega^{-s_2} V(r_3) + dV(2,r_4) + eP(r_5) + fM_{s_3}(r_6,\eta) \]
where $a, b, c, d, e, f \in \mathbb{Z}_{\geq 0},\ s_{i}\geq 1$. Let $x_{\text{core}} = aV(r_1) + b\Omega^{s_1} V(r_2) + c\Omega^{-s_2} V(r_3)$. By definition, elements in $I_{\operatorname{max}}$ do not contribute to the core. Thus:
\begin{align*}
& \dim \operatorname{core}_{\mathfrak{X}_{\operatorname{max}}} \left( x^n \right) = \dim \operatorname{core}_{\mathfrak{X}_{\operatorname{max}}} \left( (x_{\text{core}})^n \right) \\
&= \dim \operatorname{core}_{\mathfrak{X}_{\operatorname{max}}} \sum_{n_1+n_2+n_3=n} \frac{n!}{n_1! n_2! n_3!} (aV(r_1))^{n_1} (b\Omega^{s_1} V(r_2))^{n_2} (c\Omega^{-s_2} V(r_3))^{n_3}.
\end{align*}
Applying Equation \eqref{eq:6} and noting that $(\text{proj}) \in I_{\operatorname{max}}$, we have:
\begin{align*}
& \dim \operatorname{core}_{\mathfrak{X}_{\operatorname{max}}} \left( \sum_{n_1+n_2+n_3=n} \frac{n!}{n_1! n_2! n_3!} a^{n_1} b^{n_2} c^{n_3} \Omega^{n_2s_1-n_3s_2} V(r''') \right) \\
&\geq \sum_{n_1+n_2+n_3=n} \frac{n!}{n_1! n_2! n_3!} a^{n_1} b^{n_2} c^{n_3} = (a+b+c)^n
\end{align*}
taking $\dim(V(r))=1$ and $\dim(\Omega^m V(r)) \geq 1$. Taking the $n$-th root and the limit as $n \to \infty$, we obtain $\gamma_{\max}(x) \geq a+b+c$. By the subadditivity of the gamma invariant \cite[Theorem 1.50]{Benson-Banach}, we also have $\gamma_{\max}(x) \leq a\gamma_{\max}(V(r_1)) + b\gamma_{\max}(\Omega^{s_1} V(r_2))+ c\gamma_{\max}(\Omega^{-s_2} V(r_3) )= a+b+c$. Hence, $\gamma_{\max}(x) = a+b+c$.

This linearity immediately implies $\gamma_{\max}(x + y) = \gamma_{\max}(x) + \gamma_{\max}(y)$. Furthermore, since the tensor product of the core components remains in the core (modulo projectives), the multiplicative property $\gamma_{\max}(xy) = \gamma_{\max}(x)\gamma_{\max}(y)$ follows by a similar calculation.
\end{proof}

\begin{remark}\label{rem:gamma_proj} \leavevmode
In contrast, $\gamma_{\operatorname{proj}}$ is not a ring homomorphism. As an example, while $\gamma_{\operatorname{proj}}(M_s(r, \eta)) = 2$, the tensor product $M_s(r, \alpha) \otimes M_t(r', \eta)$ for $\alpha \neq \eta$ is an element of $I_{\operatorname{proj}}$, which yields $\gamma_{\operatorname{proj}}(M_s \otimes M_t) = 0 \neq 2 \times 2$. This highlights the sensitivity of the gamma invariant to the choice of representation ideal.
\end{remark}

\section{Hopf Algebras of Finite Representation Type}\label{4}

In this section, we assume that $H$ is a finite-dimensional non-semisimple Hopf algebra of finite representation type and that $\mathcal{P}_+ = \mathcal{P}_- = I_{\operatorname{max}}$. For any $x \in R$, let $\overline{x}$ denote the corresponding element in the quotient $R_{\operatorname{max}} = R/I_{\operatorname{max}}$. According to Theorem \ref{2.3}, $R_{\operatorname{max}}$ is a transitive fusion ring. We can define the \textit{Frobenius-Perron dimension} of $\overline{x_i}$ as the maximal non-negative eigenvalue of the matrix representing left multiplication by $\overline{x_i}$. The function $\operatorname{FPdim}_{R_{\mathbb{C},\operatorname{max}}}: R_{\mathbb{C},\operatorname{max}} \to \mathbb{C}$ is obtained via linear extension from the basis $\{\overline{x_i}\}_{i \in \mathfrak{J} \setminus \mathfrak{X}_{\operatorname{max}}}$.

Following \cite{Benson-Banach}, for any element $x = \sum_{i \in \mathfrak{I}} a_i x_i \in R_{\mathbb{C}}$, we define its norm as:
\[ \|x\| = \sum_{i \in \mathfrak{I}} |a_i| \dim_{\mathbf{k}}(x_i). \]
It is straightforward to verify that $R_{\mathbb{C}}$ forms a normed algebra. Since a finite-dimensional normed space is complete, $R_{\mathbb{C}}$ is a Banach algebra \cite{Marcoux1}. Furthermore, as every finite-dimensional subspace is closed, every ideal of $R_{\mathbb{C}}$ is closed. By \cite[Lemma 2.11]{Benson-Banach}, $R_{\mathbb{C}, \mathfrak{X}}$ is isometrically isomorphic to $R_{\mathbb{C}}/\langle \mathfrak{X} \rangle_{\mathbb{C}}$. Recall that if $I$ is a closed ideal in a Banach algebra $A$, then $A/I$ is a Banach algebra under the quotient norm: $\|\overline{x}\| = \| \sum_{i \in \mathfrak{J}} a_i x_i \|_{\mathfrak{X}}=\inf_{y \in I} \|x+y\|$. Thus, $R_{\mathbb{C},\mathfrak{X}}$ is a Banach algebra with the norm:
\[ \| \sum_{i \in \mathfrak{J}} a_i x_i \|_{\mathfrak{X}} = \sum_{i \in \mathfrak{J} \setminus \mathfrak{X}} |a_i| \dim_{\mathbf{k}}(x_i). \]

The following lemma identifies the gamma invariant as a spectral radius, which is a specialized case of \cite[Theorem 3.12]{Benson-Banach}.

\begin{lemma}\label{4.1}
If $x \in R_{\geq 0}$, then $\gamma_{\mathfrak{X}}(x)$ is equal to the spectral radius of the image of $x$ in $R_{\mathbb{C},\mathfrak{X}}$.
\end{lemma}

\begin{theorem}\label{4.2}
If $x, y \in R_{\geq 0}$, then $\gamma_{\operatorname{max}}(x) = \operatorname{FPdim}_{R_{\mathbb{C}, \operatorname{max}}}(\overline{x})$. This implies the following ring homomorphism properties:
\begin{align*}
\gamma_{\max}(xy) &= \gamma_{\max}(x)\gamma_{\max}(y), \\
\gamma_{\max}(x + y) &= \gamma_{\max}(x) + \gamma_{\max}(y).
\end{align*}
\end{theorem}

\begin{proof}
If $x \in I_{\operatorname{max}}$, it follows from the definition that $\gamma_{\operatorname{max}}(x) = 0$. Let $\{\overline{x_i}\}_{i \in \mathfrak{J} \setminus \mathfrak{X}_{\operatorname{max}}}$ be the basis of $R_{\operatorname{max}}$ and let $\overline{b} = \sum_{i \in \mathfrak{J} \setminus \mathfrak{X}_{\operatorname{max}}} \overline{x_i}$. Consider the matrix $B$ of right multiplication by $\overline{b}$ relative to this basis. By Theorem \ref{2.3}, $R_{\operatorname{max}}$ is a transitive ring, which implies that $B$ is a primitive matrix with strictly positive entries. According to the Frobenius-Perron Theorem, $B$ has a unique dominant positive eigenvalue $\lambda$ with a one-dimensional eigenspace spanned by a vector $\overline{u} = \sum a_i \overline{x_i}$ where $a_i > 0$.

For any $\overline{x_j}$, we have $\overline{x_j u \vphantom{b}} \overline{b} = \lambda \overline{x_j u \vphantom{b}}$, implying $\overline{x_j}\overline{u}$ is also an eigenvector for $\lambda$. By the one-dimensionality of the eigenspace, $\overline{x_j}\overline{u} = d(\overline{x_j})\overline{u}$ for some homomorphism $d$. Since $\overline{u}$ is strictly positive, $d(\overline{x}) = \operatorname{FPdim}_{R_{\mathbb{C},\operatorname{max}}}(\overline{x})$, which coincides with $\rho(\overline{x})$ by the definition of the spectral radius. By Lemma \ref{4.1}, $\gamma_{\operatorname{max}}(x) = \operatorname{FPdim}_{R_{\mathbb{C}, \operatorname{max}}}(\overline{x})$.
\end{proof}

\begin{remark}\label{4.3}
\begin{enumerate}
\item Theorem \ref{4.2} implies that extending $\gamma_{\operatorname{max}}$ from $R_{\geq 0}$ to $R$ by additivity yields a ring homomorphism.
\item In contrast to \cite{Benson-Banach}, our theorem does not require the Green ring to be commutative; it requires only that $\mathcal{P}_+ = \mathcal{P}_- = I_{\operatorname{max}}$ and that $H$ is a Hopf algebra of finite representation type. However, we have not found any examples where $\mathcal{P}_+ \neq \mathcal{P}_-$.
\item If $I_{\operatorname{max}}=I_{\operatorname{proj}}$, the invariant $\gamma_{\operatorname{proj}}(x)$ is an algebraic integer, addressing open questions in \cite{Benson-Banach}.
\end{enumerate}
\end{remark}
To illustrate, we consider the Taft algebra $H_n(q)$ for an integer $n \ge 2$ and a primitive $n$-th root of unity $q$. Recall that $H_n(q)$ is an $n^2$-dimensional Hopf algebra generated by two elements $g$ and $h$. Up to isomorphism, there are exactly $n^2$ indecomposable $H_n(q)$-modules, which are given by the set $\{M(l,r) \mid 1 \leqslant l \leqslant n, r \in \mathbb{Z}_n\}$. Notably, $\dim_{\mathbf{k}}(M(l,r))=l$ and the projective modules are precisely those with $l = n$. Furthermore, the Green ring $r(H_n(q))$ is commutative and is generated by the classes $[M(1,-1)]$ and $[M(2,0)]$. To explicitly compute the multiplication within $r(H_n(q))$, we recall the tensor products of these indecomposable modules (see \cite{Chen3}).

\begin{lemma} \label{4.4}
Let $1 \leqslant l, l' \leqslant n$ and $r, r' \in \mathbb{Z}_n$, with $l_0 = \min\{l, l'\}$ and $l_1 = \max\{l, l'\}$. We then have the following tensor product isomorphisms for indecomposable $H_n(q)$-modules $M(l,r)$:
\begin{enumerate}
    \item[(1)] If $l + l' \leqslant n$, then \vspace{-\baselineskip}
    \begin{equation} \tag{4.1} \label{eq:4.1}
        M(l, r) \otimes M(l', r') \cong \bigoplus_{i=1}^{l_0} M(|l - l'| - 1 + 2i, r + r' + i - l_0).
    \end{equation}

    \item[(2)] If $l + l' > n$, then \vspace{-\baselineskip}
    \begin{equation} \tag{4.2} \label{eq:4.2}
    \begin{split}
        M(l, r) \otimes M(l', r') &\cong \Biggl( \bigoplus_{i=1}^{l+l'-n} M(n, r + r' + 1 - i) \Biggr) \\
        &\quad \oplus \Biggl( \bigoplus_{i=1}^{n-l_1} M(|l - l'| - 1 + 2i, r + r' + i - l_0) \Biggr).
    \end{split}
    \end{equation}
\end{enumerate}
Note that the second direct sum vanishes when $n - l_1 < 1$. In particular, it follows directly from \eqref{eq:4.1} and \eqref{eq:4.2} that
\begin{equation}
M(1, r) \otimes M(l, r') \cong M(l, r+r'). \tag{4.3} \label{eq:4.3}
\end{equation}
\end{lemma} 

\begin{lemma} \label{4.5}
For the Green ring $r(H_n(q))$, we have
$$I_{\operatorname{proj}}=I_{\operatorname{max}}=\mathbb{Z}\{M(n,r) \mid r \in \mathbb{Z}_n\}.$$
\end{lemma}
\begin{proof}
By a straightforward calculation, we obtain the isomorphism $M(l,r)^{*} \cong M(l,l - r - 1)$. It then follows from the equations \eqref{eq:4.1} and \eqref{eq:4.2} that the trivial module appears as a direct summand in the tensor product decomposition of $M(l,r) \otimes (M(l,r))^*$, unless $l=n$. Therefore, we have $\mathcal{P}_+ = \mathcal{P}_- =I_{\operatorname{proj}}=I_{\operatorname{max}}$.
\end{proof}

Consequently, the Taft algebra $H_n(q)$ possesses a unique representation ideal $I_{\operatorname{max}}$, making $R_{\operatorname{max}}$ its stable Green ring $r_{st}(H_n(q))$. We next calculate the gamma invariants of the standard bases of the Green ring $r(H_n(q))$. 
\begin{theorem} \label{4.22}
For the Taft algebra $H_n(q)$, we have
\[ \gamma_{\operatorname{max}}([M(l,r)]) = \frac{\sin\left(\frac{l\pi}{n}\right)}{\sin\left(\frac{\pi}{n}\right)}, \]
for all $1 \leqslant l \leqslant n$ and $r\in \mathbb{Z}_n$.
\end{theorem}
\begin{proof}
We first evaluate $\gamma_{\operatorname{max}}([M(2,0)])$. Let $M$ be the transpose of the left multiplication matrix of $[M(2,0)]$ in $r(H_n(q))$, whose characteristic polynomial is known from \cite{Cao1}. Since $M(n,r)$ ($r \in \mathbb{Z}_n$) are the only indecomposable projective modules, the left multiplication matrix in $r_{st}(H_n(q))$ is simply the transpose of $M'$, the submatrix formed by deleting the last $n$ rows and columns of $M$. 
\[
\renewcommand{\arraycolsep}{3pt}
Z = \begin{pmatrix}
0      & 0      & 0      & \cdots & 0      & 1      \\
1      & 0      & 0      & \cdots & 0      & 0      \\
0      & 1      & 0      & \ddots & 0      & 0      \\
\vdots & \vdots & \ddots & \ddots & \ddots & \vdots \\
0      & 0      & 0      & \ddots & 0      & 0      \\
0      & 0      & 0      & \cdots & 1      & 0
\end{pmatrix}.
\]%
\noindent Based on the approach in \cite{Cao1, Li1}, the characteristic polynomial of $M'$ is 
\[
\prod_{\substack{0 \le k \le n-1 \\ 1 \le j \le n-1}} (\lambda - \sigma_{k,j}),
\]
where $\sigma_{k,j} = 2\eta^k \cos \frac{j\pi}{n}$ and $\eta^k=\cos \frac{k\pi}{n} + i\sin \frac{k\pi}{n}$. The Frobenius-Perron dimension of $[M(2,0)]$ in $r_{st}(H_n(q))$ is therefore $2 \cos \frac{\pi}{n}$, which by Theorem \ref{4.2} implies $\gamma_{\operatorname{max}}([M(2,0)])= 2 \cos \frac{\pi}{n}$.

Next, by Equation \eqref{eq:4.3}, we have $[M(1,r)][M(1,r')] = [M(1,r+r')]$, which implies that $\gamma_{\operatorname{max}}([M(1,r)]) = 1$. For $3 \leqslant l \leqslant n-1$, Equations \eqref{eq:4.1} and \eqref{eq:4.3} imply that $[M(l,0)] = [M(2,0)][M(l-1,0)] - [M(1,n-1)][M(l-2,0)]$. This relation also holds for $l=n$. Since $\gamma_{\operatorname{max}}$ is a ring homomorphism by Remark \ref{4.3}, we obtain
$$\gamma_{\operatorname{max}}([M(l,0)]) = 2\cos \frac{\pi}{n} \gamma_{\operatorname{max}}([M(l-1,0)]) - \gamma_{\operatorname{max}}([M(l-2,0)]).$$
Using the initial value $\gamma_{\operatorname{max}}([M(2,0)])= 2 \cos \frac{\pi}{n}$, the identity $\sin(l\theta) + \sin((l-2)\theta) = 2\cos(\theta)\sin((l-1)\theta)$ for $\theta = \frac{\pi}{n}$ and induction on $l$, we have
$$\gamma_{\operatorname{max}}([M(l,0)]) = \frac{\sin(\frac{l\pi}{n})}{\sin(\frac{\pi}{n})}.$$

Finally, Equation \eqref{eq:4.3} combined with the multiplicativity of $\gamma_{\operatorname{max}}$ gives 
$$\gamma_{\operatorname{max}}([M(l,r)]) = \gamma_{\operatorname{max}}([M(1,r)])\gamma_{\operatorname{max}}([M(l,0)]) = \gamma_{\operatorname{max}}([M(l,0)]).$$
This establishes $\gamma_{\operatorname{max}}([M(l,r)]) = \frac{\sin(\frac{l\pi}{n})}{\sin(\frac{\pi}{n})}$.
\end{proof}

\newpage
\centerline{ACKNOWLEDGMENTS}

This article is supported by the Jiangsu Government Scholarship for Overseas Studies, the BOF of Hasselt University and NSFC (Grant No. 12371041). The first author would like to thank Prof. Huixiang Chen for his invaluable comments.

\end{document}